\documentclass[11pt]{amsart}

\usepackage{amssymb}
\usepackage{amsmath,amsthm,geometry,enumerate}
\usepackage[dvips]{graphicx}
\usepackage{psfrag}
\usepackage{amscd}
\usepackage{xspace}
\usepackage[all]{xy}

\DeclareMathOperator{\pic}{Pic}

\DeclareMathOperator{\Hom}{Hom}

\DeclareMathOperator{\rank}{Rk}

\DeclareMathOperator{\ord}{ord}

\DeclareMathOperator{\chara}{char}

\newtheorem{theorem}{Theorem}
\newtheorem{proposition}{Proposition}
\newtheorem{corollary}{Corollary}
\newtheorem{lemma}{Lemma}
\newtheorem{conjecture}{Conjecture}
\theoremstyle{definition}
\newtheorem{definition}{Definition}

\newtheorem{remark}{\textbf{Remark}}

\newtheorem{example}{\textbf{Example}}

\begin{document}

\date{}

\title[Moduli of abelian covers of elliptic curves]{Moduli of abelian covers of elliptic curves}


\author{Nicola Pagani}
\address{Department of Mathematical Sciences, The University of Liverpool, United Kingdom\\}
\email{pagani@liv.ac.uk}

\begin{abstract} For any finite abelian group $G$, we study the moduli space of abelian $G$-covers of elliptic curves, in particular identifying the irreducible components of the moduli space. We prove that, in the totally ramified case, the moduli space has trivial rational Picard group, and it is birational to the moduli space $M_{1,n}$, where $n$ is the number of branch points. In the particular case of moduli of bielliptic curves, we also prove that the boundary divisors are a basis of the rational Picard group of the admissible covers compactification of the moduli space.
Our methods are entirely algebro-geometric.
 \end{abstract}
\maketitle

\section{Introduction}

\subsection{Summary of our results}The purpose of this paper is to study the geometry of moduli spaces of abelian covers of elliptic curves. An abelian cover is a cover which is generically a principal $G$-bundle for $G$ a certain fixed finite abelian group; from this description abelian covers are a natural generalization of double covers. Our original motivating example is indeed the moduli space of bielliptic curves.

To describe the moduli spaces, we use the powerful language of moduli of stable maps to the classifying stack $BG$, first discovered in \cite{abvis2} and explained in more detail in \cite{acv} by Abramovich, Corti and Vistoli. In general, the moduli spaces that we study here are slightly different from the one na\"ively described in the first paragraph, in that the action of $G$ on the covering curve $\tilde{C}$ is fixed, and the branch points are ordered (moreover, we also allow ourselves to mark and order other points of the quotient $\tilde{C}/G$). These moduli spaces admit a forgetful map, which only remembers the covering curve, to closed subspaces of the moduli spaces of curves.

For example, in the case of the moduli space of bielliptic curves, our results are naturally presented for moduli whose objects parametrize tuples $(\tilde{C}, \alpha, x_1, \ldots, x_n)$, where $\tilde{C}$ is a genus-$g$ curve, $\alpha$ is a distinguished involution such that $\tilde{C}/\langle \alpha \rangle$ is a genus-$1$ curve and $\{x_1, \ldots, x_n\}$ is the branch locus of the map $\tilde{C} \to \tilde{C}/\langle \alpha \rangle$. By quotienting this moduli space by the action of the symmetric group $\mathfrak{S}_n$ on the points $x_i$, we retrieve the moduli of bielliptic curves without an ordering of the branch points. On the other hand, this moduli space admits a finite map to the moduli space where the involution is only required to exist and it is not fixed (for genus $g \geq 6$ there is at most a unique bielliptic involution, so such a finite map is in fact an isomorphism), the latter is naturally a codimension-$(g-1)$ moduli subscheme of $\mathcal{M}_g$.

We now summarize the main results of our work. In Theorem \ref{fondamentale}, we give a geometric description of the irreducible components of the moduli spaces of abelian $G$-covers of elliptic curves. From this description, we are able to deduce further results for those components that parametrize \emph{totally ramified} covers. In fact, in Section \ref{birationalgeo}, we prove that each component of the moduli space parametrizing totally ramified abelian $G$-covers is birational to $\mathcal{M}_{1,n}$ where $n$ is the number of marked points (in general, a superset of the set of branch points). In Theorem \ref{vanishes}, we exploit our description to prove the vanishing of the rational Picard group of such components of the moduli space, and to find the number of independent relations among the boundary divisors of the admissible covers compactification. In the last section we make our results more explicit in the case of the rational Picard group of the moduli spaces of admissible bielliptic curves, see Corollaries \ref{linindip} and \ref{numerico}.

The Picard number of the compact moduli of hyperelliptic curves $\overline{\mathcal{H}}_{g}$ is $g$ \cite{ch}. In Corollary \ref{linindip}, we see that, when $g>2$, it is $2g$ for the compact moduli $\overline{\mathcal{B}}_{g}$ of bielliptic curves  (without ordering the branch). A simple analysis of the number of boundary divisors of the compact moduli of double covers of genus-$h$ curves leads us to conjecture that the Picard number of the latter moduli space grows is $(h+1)g + g \cdot o(1)$ for $g \to \infty$. 

\subsection{Literature and the state of the art}
Our moduli spaces are generalizations of moduli of hyperelliptic curves. The geometry of the latter has been thoroughly investigated; we know that they are rational (\cite{katsylo}), and we have a description of their Picard group (\cite{ch}, \cite{arsievistoli}). We address the reader to Section \ref{genus0}, in which we review some known results, relevant for this paper, regarding the geometry of moduli of abelian covers of genus-$0$ curves and their compactification.

There is a large body of work concerning the birational geometry of moduli of bielliptic curves due to Bardelli, Casnati and Del Centina; see \cite{cace} and the references therein. The rational Picard group of moduli of bielliptic curves was investigated by the author in the cases of $g=2$ \cite{pagani2} and $g=3$ \cite{paganitommasi} (with Tommasi). In \cite{faberpagani} (with Faber), we investigate the problem of enumerating bielliptic curves of genus $3$.  The geometry of the Hurwitz scheme of covers of positive genus curves is studied in \cite{ghs}, and the moduli spaces that we are studying are closed subspaces of these Hurwitz schemes, of increasingly high codimension as the order of $G$ increases.

 As we are treating the case of curves with trivial canonical bundle, when $G$ is cyclic the objects we are considering are twisted higher level Prym or Spin curves of genus $1$. The literature on such moduli spaces and on the problem of compactifying them is vast; we address the reader to the introduction of \cite{ccc} and the references therein. Our main results are similar to those obtained by Bini-Fontanari \cite{bf1}, \cite{bf2}; we remark however that there is no intersection between our results and theirs, as their interest is focused on the component of the moduli spaces that generically parametrizes \'etale double covers. Let us remark furthermore that our results are entirely algebro-geometric and hold more generally for $\chara k \nmid |G|$.

 Our approach of treating moduli of decorated elliptic curves by means of moduli of stable maps to a stack is similar in flavour to the works \cite{petersen} and \cite{niles}; our motivation and main results are orthogonal to theirs.

The results proven in this paper have a non-trivial intersection with a recent paper by Poma-Talpo-Tonini. In \cite{ptt}, they compute the integral Picard group of the stack of uniform $G$-covers, without an ordering of the branch locus, without any restriction on the genus of the target curve. Their preprint appeared approximately one year after the preprint version of this paper, and the methods are very different.

 \subsection{Further directions}
We propose two future directions to further motivate this work.

Firstly, the moduli spaces $\overline{\mathcal{M}}_{1,n}(BG)$ describe the puntual picture of the moduli spaces of stable maps to any (abelian) orbifold, and our work is therefore a first step towards studying genus-$1$ Gromov-Witten theory for orbifolds. The study of the Gromov-Witten theory with genus-$1$ source curve has seen a lot of progress in the latest years, due to the work of Zinger and his coauthors, see \cite{zinger} and the references therein.

Secondly, we know that, on the totally ramified components of the compactified moduli spaces of abelian covers of elliptic curves, the pull-back of $\lambda$ can be expressed in terms of boundary divisors. In a forthcoming work we plan to find the explicit such expression, and to study the slope of $1$-dimensional families, as Cornalba and Harris have done in \cite{ch} for hyperelliptic families, in the hope of finding new lower bounds for the slope of divisors on $\overline{\mathcal{M}}_g$.  The attempt to find new lower bounds for the slope via moduli of covers of genus-$1$ curves has already been pursued by D.Chen \cite{chen}, again however the $1$-dimensional families he considers do not enter in our framework as they do not parametrize abelian covers. Barja has studied bielliptic fibrations in \cite{barja} and he has proven that the sharp upper bound for the slope is $8$.

\subsection{Notation} \label{notation}
In this paper, $G$ will always be a fixed finite abelian group. We work over an algebraically closed field $k$ of characteristic not dividing $|G|$, the order of the group $G$. With $G^{\vee}$ we denote the group of characters $\Hom(G, \mathbb{G}_m)$, where $\mathbb{G}_m$ is the group of invertible elements of $k$.

We fix once and for all compatible generators $\omega_l$ of the primitive $l$-roots of $1$, for each $l$ dividing $|G|$, where compatible means $\omega_l^{l/s}=\omega_s$ if $s$ divides $l$. The choice of a basis $\langle g_1\rangle \oplus \ldots \oplus \langle g_k \rangle =G$ determines the choice of a dual basis $\langle g_1^{\vee}\rangle \oplus \ldots \oplus \langle g_k^{\vee} \rangle =G^{\vee}$, where
\begin{displaymath} g_j^{\vee}(g_i)= \begin{cases} \omega_{|\langle \chi_i\rangle|} & j=i \\ 1 & j \neq i. \end{cases} \end{displaymath}
Conversely, the choice of a basis of $G^{\vee}$ determines a dual basis of $G$, our assumption thus gives a canonical isomorphism $G \to G^{\vee}$.


Throughout the paper we use the letter $g$, respectively $h$, for the genus of the covering curve $\tilde{C}$, respectively for the genus of the covered curve $C$.

\section{Moduli spaces of abelian covers of curves}

\subsection{General setting}
The main object of study of this paper are moduli spaces of abelian covers of smooth curves of fixed genus. We fix a finite abelian group $G$, and  integers $h,n$ such that $2h-2+n>0$. For convenience, we introduce an ordering of the set of branch points; moreover, we admit the possibility of marking further points inside the covered curve. The reference for this section is \cite{acv}, a specialization of the more general theory developed in \cite{abvis2}.

\begin{definition} \label{plain} We define the moduli stack $\mathcal{M}_{h,n}(B G)$ of abelian $G$-covers, whose geometric points are tuples $( (C,x_1, \ldots, x_n), \ p \colon \tilde{C} \to C )$, satisfying the following.
\begin{enumerate}
\item The datum $(C,x_1, \ldots, x_n)$ is an irreducible, $n$-pointed smooth projective curve of genus $h$.
\item The curve $\tilde{C}$ is smooth and projective (but not necessarily irreducible).
\item The morphism $p \colon \tilde{C} \to C$ is a cover (it is finite, flat, of finite presentation), and $G$ acts on $\tilde{C}$ so that $p$ is $G$-invariant.
\item The restriction \begin{displaymath}p^{gen} \colon p^{-1}(C \setminus \{x_1, \ldots, x_n\}) \to C \setminus \{x_1, \ldots, x_n \}\end{displaymath} is an \'etale $G$-torsor.
\end{enumerate}
\end{definition}

There is a modular compactification made with admissible $G$-covers.

\begin{definition} \label{admissible} We define the moduli stack $\overline{\mathcal{M}}_{h,n}(BG)$ of admissible $G$-covers, whose geometric points are tuples $( (C,x_1, \ldots, x_n), \ p \colon \tilde{C} \to C )$, satisfying the following.
\begin{enumerate}
\item The datum $(C,x_1, \ldots, x_n)$ is a stable, $n$-pointed curve of arithmetic genus $h$.

\item The curve $\tilde{C}$ is nodal and projective (but not necessarily irreducible, nor stable).

\item The morphism $p \colon \tilde{C} \to C$ is a cover, and $G$ acts on $\tilde{C}$ so that $p$ is $G$-invariant.

\item Let $C^{sm}$ be the open locus where $C$ is smooth, then the restriction \begin{displaymath}p^{gen} \colon p^{-1}(C^{sm} \setminus \{x_1, \ldots, x_n\}) \to C^{sm} \setminus \{x_1, \ldots, x_n \}\end{displaymath} is an \'etale $G$-torsor. The image via $p$ of each node of $\tilde{C}$ is a node of $C$.
\item \emph{Smoothability}: for any node $q$ of $\tilde{C}$, the stabilizer of the $G$-action at $q$ acts with inverse characters on the tangent spaces at the two branches of $q$.
\end{enumerate}
\end{definition}

\begin{remark} \label{stackyremark} In Definitions \ref{plain} and \ref{admissible} it is standard to define the morphisms to be the cartesian diagrams. For later convenience (Theorem \ref{fondamentale}) in this paper we always work with the \emph{rigidified} moduli stacks (see \cite[Section 5]{acv}). For example, morphisms of $\mathcal{M}_{h,n}(B G)$ are morphisms between target curves that \emph{admit} a $G$-equivariant lift.

\end{remark}

The moduli stack $\mathcal{M}_{h,n}(BG)$ is open and dense in $\overline{\mathcal{M}}_{h,n}(BG)$; the latter is a smooth proper Deligne--Mumford stack with projective coarse moduli space (see \cite{acv}, \cite{abvis2}).  We will mainly be working in the case when $h=1$; we will investigate geometric properties of the moduli spaces such as irreducibility, Kodaira dimension and rational Picard group.

\subsection{Moduli stacks via abelian covers theory} \label{abeliancovers}
This section is inspired by \cite{pardini}. We take the chance here to fix our notation for abelian covers. The reason for introducing the theory only in dimension one is twofold. We can conveniently work in moduli since at the $1$-dimensional level we have a moduli space parametrizing couples $(X,D)$ where $X$ is a variety and $D$ is a smooth effective divisor on it. Moreover, we can avoid issues of normalization (cf. \cite[Section 3]{pardini}) of the total space of the cover.

 To any object of $\mathcal{M}_{h,n}(BG)$, we can associate discrete invariants besides $h$ and $n$. There are $n$ evaluation maps (see \cite[Section 4.4]{agv2}; they write $\mathcal{K}$ in place of $\mathcal{M}$)
\begin{displaymath}
ev_j \colon \mathcal{M}_{h,n}(BG) \to I(BG),
\end{displaymath}
where $I(BG)$ is the inertia stack (\cite[Section 3]{agv2}) of $BG$:
\begin{equation} \label{ibg}
I(BG)= \{(H, \psi), \  H \ \textrm{a cyclic subgroup of } G, \ \psi  \textrm{ a generator of } H^{\vee} \}.
\end{equation}
Geometrically, $ev_j$ of a $G$-cover is $(H_j,\psi_j)$, where $H_j$ is the stabilizer of the action of $G$ on any of the fibers over $x_j$ (a cyclic group), and $\psi_j$ is the character of the induced action of $H_j$ on the cotangent space on any of the fibers over $x_j$. The stabilizer $H_j$ is the trivial group exactly when $x_j$ is not in the branch locus of the $G$-cover.

From now on, having a convention on the roots of unity (see Section \ref{notation}) we can naturally identify $G\setminus \{ 0\}$ with $I(BG)$, by associating $h_j$ with $H_j:= \langle h_j \rangle$ and $\psi_j(h_j):= \omega_{|\langle h_j \rangle|}$. Once the discrete data $h_1, \ldots, h_n$ are fixed, we can ask if the open and closed component of $\mathcal{M}_{h,n}(BG)$
\begin{equation} \label{modulistack}
\mathcal{M}_{h,n}(BG)\left( (h_1), \ldots, (h_n) \right):= \bigcap_{i=1}^n \ ev_j^{-1} (h_j)
\end{equation}
is irreducible. To accomplish this, we will express the moduli in an alternative way, building on Pardini's description of abelian covers \cite{pardini}.

We incidentally observe that the genus of the covering curve is a redundant invariant. Indeed, it is expressible in terms of the other invariants via the Riemann-Hurwitz formula:
\begin{equation} \label{rh}
2g-2=|G|(2h-2)+ \sum_{j=1}^n |G|(1-|\langle h_j \rangle|^{-1}).
\end{equation}

Let $((C,x_1,\ldots,x_n), p \colon \tilde{C} \to C)$ be an object of \eqref{modulistack}, namely an abelian $G$-cover. The vector bundle $p_*\mathcal{O}_{\tilde{C}}$ naturally splits as a sum of line bundles $L_{\chi}^{-1}$ for $\chi \in G^{\vee}$. For any $\chi \in G^{\vee}$, we define $a_{\chi}^j$ as the smallest natural number such that \begin{equation} \label{aij} \chi_{| \langle h_j \rangle|}= \psi_j^{a_{\chi}^{j}}.\end{equation} Then, for $\chi_1, \chi_2 \in G^{\vee}$, and for $j \in \{1, \ldots, n\}$, we define the natural number \begin{displaymath}\epsilon_{\chi_1, \chi_2}^{j}:=\lfloor \frac{a_{\chi_1}^{j}+ a_{\chi_2}^{j}}{|\langle h_j \rangle|}  \rfloor= \begin{cases} 0 &\textrm{if } a_{\chi_1}^j + a_{\chi_2}^j< |\langle h_j \rangle|, \\ 1 &\textrm{otherwise.}\end{cases}\end{displaymath}  With this notation, the multiplication in the $\mathcal{O}_C$-algebra $p_*\mathcal{O}_{\tilde{C}}$ induces isomorphisms
\begin{equation} \label{compa}
\phi_{\chi_1, \chi_2} \colon L_{\chi_1} \otimes L_{\chi_2} \cong L_{\chi_1 \chi_2} \otimes \bigotimes_{j=1}^n \mathcal{O}_C\left(\epsilon_{\chi_1, \chi_2}^{j} x_j\right)  \quad \textrm{for all } \chi_1,\chi_2 \in G^{\vee}.
\end{equation}
Pardini proves in \cite[Theorem 2.1]{pardini} that conversely, if $(C, x_1,\ldots, x_n)$ is given, the abelian $G$-cover $p\colon \tilde{C} \to C$ can be reconstructed up to a $G$-equivariant isomorphism from the set of data that appear in \eqref{compa}. That such a reconstruction is possible in the relative case over a scheme follows from the more general recent result \cite[Theorem 4.40]{tonini} by Tonini.

\begin{proposition}  \label{equivalentdesc} The moduli stack \eqref{modulistack} is isomorphic to the moduli stack of objects \begin{displaymath}((C, x_1, \ldots, x_n), \{L_{\chi}\}_{\chi \in G^{\vee}}, \{\phi_{\chi_1, \chi_2}\}_{\chi_1, \chi_2 \in G^{\vee}}).\end{displaymath} Here $(C, x_1, \ldots, x_n)$ is a smooth $n$-pointed curve of genus $h$, the $L_{\chi}$ are line bundles on $C$, the $\phi_{\chi_1, \chi_2}$ are isomorphisms of line bundles and the data satisfy relations \eqref{compa}. The morphisms are morphisms of pointed curves that admit  lifts at the level of line bundles, such that the lifts are compatible with the constraints \eqref{compa} (see Remark \ref{stackyremark2}).
\end{proposition}

\begin{remark} \label{stackyremark2} Here is a more explicit description of the morphisms of the second moduli stack mentioned in Proposition \ref{equivalentdesc} (see also Remark \ref{stackyremark}). Let $(C, x_i, L_{\chi},\phi_{\chi_1, \chi_2})$ and $({C}', {x}_i', {L}_{\chi}', {\phi}'_{\chi_1, \chi_2})$ be two objects. Then a morphism is a morphism of pointed curves $\sigma \colon (C, x_i)\to ({C}', {x}'_i)$, such that for all $\chi \in G^{\vee}$, there exists  $\tau_{\chi} \colon \sigma^*{L'}\to L$ that makes the following diagram commute:
\begin{equation} \label{diagrammone}
\xymatrix{ \sigma^*({L'}_{\chi_1} \otimes {L'}_{\chi_2})  \ar[rr]^{\hspace{-1.3cm}\sigma^*\left({\phi'}_{\chi_1, \chi_2}\right)} \ar[d]^{\tau_{\chi_1} \otimes \tau_{\chi_2}} & & \sigma^*\left({L'}_{\chi_1 \chi_2} \otimes \bigotimes_{j=1}^n \mathcal{O}_{{C'}}\left(\epsilon_{\chi_1, \chi_2}^{j} {x'}_j\right)\right) \ar[d]^{\sigma^* \left(\tau_{\chi_1 \chi_2}\right) \otimes \gamma} \\
L_{\chi_1} \otimes L_{\chi_2} \ar[rr]^{\hspace{-1.3cm}\phi_{\chi_1, \chi_2}} && L_{\chi_1 \chi_2} \otimes \bigotimes_{j=1}^n \mathcal{O}_C\left(\epsilon_{\chi_1, \chi_2}^{j} x_j\right)}
\end{equation}
where $\gamma \colon \mathcal{O}_{{C'}}\left(\epsilon_{\chi_1, \chi_2}^{j} {x'}_j\right) \to \mathcal{O}_C\left(\epsilon_{\chi_1, \chi_2}^{j} x_j\right) $ is the morphism induced by $\sigma$.
A different definition for a morphism between the two families would be to fix a pair $(\sigma, \tau)$ as above. As we have already observed in Remark \ref{stackyremark}, the resulting moduli stack is a $G$-banded gerbe over the stack we are working with.
\end{remark}

\begin{remark} \label{etaleremark} In the case of \'etale $G$-covers, Proposition \ref{equivalentdesc} simply says that principal $G$-bundles over a fixed $C$ are classified topologically, up to $G$-equivariant isomorphisms, by $\Hom(G^{\vee}, Pic^0(C))$. Over $\mathbb{C}$, the latter abelian group is isomorphic to
\begin{displaymath}
 \Hom(H_1(C, \mathbb{Z}), G) \cong H^1(C, G),
\end{displaymath}
where the cohomology groups are taken in the analytic topology of $C$.
\end{remark}

There is a simpler description of the moduli stack \eqref{modulistack} if one pays the price of fixing a basis for $G$, \emph{i.e.} $G= \langle g_1 \rangle \oplus \ldots \oplus \langle g_l \rangle$, or equivalently, a dual basis for $G^{\vee}= \langle \chi_1\rangle \oplus \ldots \oplus \langle \chi_l \rangle$. Let us pose $d_i:=\ord (g_i)$. If the ramifications $h_1, \ldots, h_n$ are fixed, the homomorphism induced by the componentwise inclusion $\bigoplus \langle h_j \rangle \to G$ can be written, in the chosen bases, as
\begin{equation} \label{lambdaij}h_j = \sum_i \lambda_{i j} g_i, \quad \lambda_{i j}:= \frac{d_i \cdot a_{g_i^{\vee}}^{j}}{|\langle h_j \rangle|}, \end{equation}
where $a_{\chi}^j$ was defined in \eqref{aij}.

The line bundles $L_{\chi}$ that appear in Proposition \ref{equivalentdesc} can be reconstructed from the line bundles $L_{\chi_1}, \ldots, L_{\chi_l}$ by using \eqref{compa}. Tensoring $L_{\chi_i}$ with itself $d_i$ times, and replacing the indices $\chi_i$ with $i$, from \eqref{compa} we arrive at the compatibility conditions
\begin{equation}\label{reducedeq}
\phi_i \colon L_i^{\otimes d_i} \cong \bigotimes_{j=1}^n \mathcal{O}_C\left(\lambda_{i j} x_j \right), \quad i=1 \ldots, l.
\end{equation}
In \cite{pardini}, the set of line bundles $\{ L_i\}_{1 \leq i \leq l}$ takes the name of \emph{reduced building data} for the $G$-cover. After having chosen a basis for $G$, we can recast Proposition \ref{equivalentdesc}.

\begin{proposition}  \label{equivalentred} The moduli stack \eqref{modulistack} is isomorphic to the moduli stack whose objects are \begin{displaymath}((C, x_1, \ldots, x_n), \{L_i\}_{1 \leq i \leq l}, \{\phi_i\}_{1 \leq i \leq l} ).\end{displaymath} Here $(C, x_1, \ldots, x_n)$ is a smooth $n$-pointed curve of genus $h$, the $L_{i}$ are line bundles on $C$, the $\phi_i$ are isomorphisms of line bundles, and the data satisfy relations \eqref{reducedeq}. Morphisms are as in Proposition \ref{equivalentdesc} (cf. Remark \ref{stackyremark2}).
\end{proposition}

\begin{proof} The proof of \cite[Proposition 2.1]{pardini} shows that the moduli stacks introduced in Proposition \ref{equivalentdesc} and in this Proposition are isomorphic (the argument that she gives over a geometric point applies verbatim to prove the relative case).
\end{proof}

\subsection{Totally ramified abelian covers}

Let us fix a finite abelian group $G$ and $n$ ramifications $h_1, \ldots, h_n \in G$. The moduli stack \eqref{modulistack}
is not irreducible in general. It follows from Proposition \ref{equivalentdesc} that the forgetful map
\begin{equation} \label{torsor}
\mathcal{M}_{h,n}(BG, (h_1), \ldots, (h_n)) \to \mathcal{M}_{h,n}
\end{equation}
describes the former as a torsor whose structural group scheme is given locally in  the point $(C, x_1, \ldots, x_n)$ of  $\mathcal{M}_{h,n}$ by $\Hom(G^{\vee}, \pic^0(C))$. In $\chara k= 0$, a classical argument of monodromy, originally due to Hurwitz and Clebsch, sheds some light in the study of the irreducible components of the torsor \eqref{torsor}. Indeed, let us work over $\mathbb{C}$. 
Then the monodromy action on the set of data of Proposition \ref{equivalentred}, when $x_1=x_1(t)$ moves in a closed loop that identifies the homology class $\xi \in H_1(C, \mathbb{Z})$, is given by
\begin{displaymath}
L_{i} \to L_{i} \otimes M_i,
\end{displaymath} where $M_i$ is the $d_i$-th torsion point in $\pic^0(C)$ corresponding to $\frac{\lambda_{i j}} {d_i} \xi $.

\begin{example} Suppose that $h_j=0$ for all $1 \leq j \leq n$, \emph{i.e.} we are studying \'etale $G$-covers of smooth genus-$h$ curves with $n$ distinct marked points on the target curve. Then the stack \eqref{modulistack} is irreducible if and only if $G$ is trivial or $h=0$. In the latter case, the moduli space of \eqref{modulistack} and of $\mathcal{M}_{h,n}$ are the same.
\end{example}

We now introduce a hypothesis for $G$-covers, on the opposite spectrum of \'etale. 
\begin{definition} \label{totram} An abelian $G$-cover is \emph{totally ramified} when the map
$
\bigoplus \langle h_j \rangle \to G,
$
induced from the componentwise inclusion is surjective.
\end{definition}
We observe that the hypothesis of total ramification guarantees that each object of \eqref{modulistack} corresponds to a connected cover, and that moreover, through the monodromy argument introduced above, the resulting moduli space \eqref{modulistack}, in the totally ramified case, is either irreducible or empty. From now on, we denote by $R$ and $E$ respectively the image and the cokernel of the componentwise inclusion $\bigoplus \langle h_j \rangle \to G$. With this notation, the sequence of abelian groups \begin{equation} 0 \to R \to G \to E \to 0\end{equation} is exact. Each abelian $G$-cover $\tilde{C} \to C$ factors as the composition of a totally ramified $R$-cover $\tilde{C} \to \tilde{C}':=\tilde{C}/R$ followed by an \'etale $E$-cover $\tilde{C}' \to C$. This splitting gives a map of moduli stacks
\begin{equation} \label{etale}
\mathcal{M}_{h,n}(BG, (h_1), \ldots, (h_n)) \to \mathcal{M}_{h,n}(BE, 0),
\end{equation} whose fibers are irreducible by the above monodromy argument. Thus, at least in $\chara k=0$, the irreducible components of \eqref{modulistack} correspond to the components of $\mathcal{M}_{h,n}(BE, 0)$.

We now reformulate Proposition \eqref{equivalentred}, by splitting the equations \eqref{reducedeq}, according to the factorization of the cover discussed in the above paragraph. This splitting depends on the choice of a basis of $E^{\vee}= \langle e_1 \rangle \oplus \ldots \oplus \langle e_{l} \rangle$, and of elements $g_1', \ldots, g_{k}' \in G^{\vee}$ whose canonical restriction forms a basis of $R^{\vee}$. By duality, the latter gives a basis $r_1, \ldots, r_{k}$ of $R$. We can fix unique integers $0 \leq c_{ij}< o(e_j)$, such that $r_i \cdot g_i'= \sum_j c_{ij} e_j^{\vee}$, and unique integers $0 \leq b_{ij}< o(r_i)$ such that $h_j= \sum b_{ij} r_i$. In the same way in which one proves the equivalence of the solutions of \eqref{compa} and of \eqref{reducedeq}, we can also prove the equivalence of the solutons of \eqref{compa} and the solutions $(C, x_j, L_i', L_i'', \phi_i', \phi_i'')$ of

\begin{equation} \label{reducedsplit}
\begin{cases}
\phi_i'\colon L_i'^{o (r_i)} \cong \mathcal{O}_C\left(\sum_{j=1}^n b_{ij} x_j \right) \otimes \bigotimes_{j=1}^{l} L_j''^{\otimes c_{ij}}, & 1 \leq i \leq k,\\
\phi_i''\colon L_i''^{o(e_i)} \cong \mathcal{O}_C & 1 \leq i \leq l,
\end{cases}
\end{equation}
where $L_i'$ stands for $L_{g_i'}$ and $L_i''$ stands for $L_{e_i}$. The second set of equations defines the \'etale cover $\tilde{C}' \to C$, and the pull-back to $\tilde{C}'$ of the equations \begin{displaymath} \phi_i'\colon L_i'^{\otimes r_i} \cong \mathcal{O}_C\left(\sum_{j=1}^n b_{ij} x_j \right)\end{displaymath} defines the totally ramified cover $\tilde{C} \to \tilde{C}'$. We summarize the above discussion in the following proposition.

\begin{proposition} \label{equivalentred2} The moduli stack \eqref{modulistack} is isomorphic to the moduli stack whose objects are $(C, x_j, L_i', L_i'', \phi_i', \phi_i'')$, where $(C, x_1, \ldots, x_n)$ is a smooth $n$-pointed curve of genus $h$, the $L_{i}', L_i''$ are line bundles on $C$, the $\phi_i', \phi_i''$ are isomorphisms of line bundles, and the data satisfy relations \eqref{reducedsplit}. Morphisms are as in Propositions \ref{equivalentdesc} and \ref{equivalentred}.
\end{proposition}

\subsection{An example: the genus-$0$ case} \label{genus0}
Let us study the case of $\mathcal{M}_{0,n}(BG)$. The main fact that we be use here is that a degree $0$ line bundle over a smooth genus-$0$ curve is always trivial. It follows then from Proposition \ref{equivalentdesc} that the moduli stack $\mathcal{M}_{0,n}(BG)$ has $\mathcal{M}_{0,n}$ as its coarse moduli space. It is also true that $\overline{\mathcal{M}}_{0,n}$ is the coarse moduli space of $\overline{\mathcal{M}}_{0,n}(BG)$. If $G$ is cyclic, this follows from the main theorem in \cite[p. 7]{bayercadman}. The general case then follows after decomposing the finite abelian group $G$ as a direct sum of cyclic groups.  Then we have:
\begin{enumerate}
\item The moduli stack $\mathcal{M}_{0,n}(BG)$ is always empty or irreducible.
\item The coarse moduli space of $\overline{\mathcal{M}}_{0,n}(BG)$ is birational to $\overline{\mathcal{M}}_{0,n}$.
\item The rational Picard groups of $\mathcal{M}_{0,n}(BG)$ are zero. (In fact, the rational Chow group of a stack equals the rational Chow group of its coarse moduli space, and the Chow group of $\mathcal{M}_{0,n}$ vanishes since the moduli space of smooth, $n$-pointed genus-$0$ curves is a complement of hyperplanes in $\mathbb{P}^{n-3}$.)
\item Moreover, the rational Picard group of $\overline{\mathcal{M}}_{0,n}(BG)$ is \emph{freely} spanned by the boundary divisors. This is because the same statement is true for $\overline{\mathcal{M}}_{0,n}$ \cite{keel}. (In fact, the cohomology ring equals the Chow ring, and the latter is generated by boundary divisors with explicit relations.)
\end{enumerate}
The task of this paper is to study extensions of these genus-$0$ statements to genus $1$. This is accomplished in Theorem \ref{fondamentale}, Corollaries \ref{birat1} and \ref{birat2}, and Theorem \ref{vanishes}. The case $G= \mathbb{Z}/2\mathbb{Z}$ is especially interesting as any double cover is automatically an abelian $\mathbb{Z}/2\mathbb{Z}$-cover. We will thus devote Section \ref{modulibiell} to giving more details for this special case.

In the genus-$0$ case, consider the quotient stack $[\overline{\mathcal{M}}_{0,2g+2}(B \mathbb{Z}/2\mathbb{Z})/\mathfrak{S}_{2g+2}]$ under the action of the symmetric group on the markings. The latter moduli stack parametrizes families of admissible hyperelliptic curves of genus $g$ and we call it $\overline{\mathcal{H}}_g$. Bogomolov and Katsylo proved that the coarse space of $\overline{\mathcal{H}}_g$ is rational (\cite{katsylo}, \cite{bogokat}). 
 The integral Chow groups of the moduli stack of hyperelliptic curves are known from a work of Fulghesu and Viviani \cite{fulghesu}, which builds on previous results of Arsie and Vistoli \cite{arsievistoli}. The latter works compute in particular the whole integral Chow groups (respectively, Picard groups) of the uniform components\footnote{A cover is \emph{uniform} when it is totally ramified and with equal ramification at all points.} of $\mathcal{M}_{0,n}(BG)$ when $G$ is finite and cyclic.

\section{Moduli spaces of abelian covers of elliptic curves}

From now on, we consider the genus $h=1$ case, with $n \geq 1$. When $(C, x_1, \ldots, x_n)$ is a smooth $n$-pointed genus $1$ curve, we also call it an elliptic curve: the origin of the group law is always conventionally assumed to be the first marked point\footnote{This choice breaks the $\mathfrak{S}_n$-symmetry.}.

For a given a choice of a finite abelian group $G$, and for fixed ramifications $h_1, \ldots, h_n$, we give an alternate description of
\begin{equation} \label{moduligenus1}
\mathcal{M}_{1,n}(BG, (h_1), \ldots, (h_n)).
\end{equation}
We will describe the irreducible components of \eqref{moduligenus1} with an entirely algebro-geometric approach. (Over $\mathbb{C}$, we already know from the discussion that immediately follows \eqref{etale}, that the components of \eqref{moduligenus1} correspond to the components of $\mathcal{M}_{1,n}(BE,0)$, which in turn correspond to certain modular curves. )

So let us review the \'etale case. The main point is that at most $2$ torsion points may be linearly independent on an elliptic curve. When $E$ is a finite abelian group,  
it is well-known that the stack $\mathcal{M}_{1,1}(BE)$ admits a natural decomposition in components, not necessarily irreducible, that are parametrized by subgroups $K$ of $E$ with $0 \leq \rank(K) \leq 2$.
\begin{equation} \label{decomp}
\mathcal{M}_{1,1}(BE) \cong \coprod_{\substack{K<E, \\ 0 \leq \rank{K} \leq 2}} \mathcal{Y}_1(K). 
\end{equation} (The components corresponding to subgroups $K$ of $E$ of rank $2$ further decompose according to the Weil (or symplectic) pairing.)

When $K \cong \{0\}$ has rank $0$, $\mathcal{Y}_1(K)$ is naturally isomorphic to $\mathcal{M}_{1,1}$. When $K \cong \mathbb{Z}/N \mathbb{Z}$ has rank $1$, $\mathcal{Y}_1(K)$ is isomorphic to the stacky modular curve $\mathcal{Y}_1(N)$; Over $\mathbb{C}$, the latter is the {stack} quotient of the upper half-plane $[\mathbb{H}/\Gamma_1(N)]$. When $K$ has rank $2$, the component $\mathcal{Y}_1(K)$ is not irreducible; if we fix an isomorphism $K \cong \mathbb{Z}/M \mathbb{Z} \times \mathbb{Z}/N \mathbb{Z}$ with $M|N$, then we have
\begin{equation}\label{decomp2}
\mathcal{Y}_1(K) \cong \coprod_{i \in \mu_M^*} (\mathcal{Y}_1(M,N),i),
\end{equation}
where now $\mathcal{Y}_1(M,N)$ is a (irreducible) \emph{stacky modular curve}, that, over $\mathbb{C}$, is just the {stack} quotient of the upper half-plane $[\mathbb{H}/\Gamma_1(M,N)]$. The element $i$ is a generating $M$-th roots of $1$ that corresponds to the value of the Weil pairing on each component of $\mathcal{Y}_1(K)$.

 We address the reader to \cite[Sections 2, 7]{diamond} for the above-mentioned results in $\chara k =0$. The irreducibility of the components of the decomposition \eqref{decomp}, \eqref{decomp2}, is established, for $\chara(k)$ not dividing $|E|$, in \cite[Chapters  3,4]{km}, see \cite[Corollary 4.7.2]{km}.

  We will also denote the $m$-th universal curve over $\mathcal{M}_{1,n}$ with $\mathcal{C}^m_{1,n}$. The following fiber product diagram defines then $\mathcal{Y}_1(K)_n^m$, where the right vertical map is the composition of the natural $m$-th universal curve map $\mathcal{C}^m_{1,n}\to \mathcal{M}_{1,n}$ with the map forgetting all points but the first one
\begin{equation} \label{diagram} \xymatrix{
 \mathcal{Y}_1(K)_n^m \ar[r] \ar[d]\ar@{}[rd]|{\square} & \mathcal{C}^m_{1,n} \ar[d]& \hspace{-0.7cm}\\
 \mathcal{Y}_1(K) \ar[r] & \mathcal{M}_{1,1}.&\\
}\end{equation}

\subsection{The geometry of the irreducible components}

\label{irred}
As in the previous section, we denote with $R$ the image of the natural componentwise inclusion map $\bigoplus \langle h_i \rangle \to G$, and $E$ the quotient group $G/R$. We address the reader to the previous section for our notation on stacky modular curves, for $\mathcal{Y}_1(K)$, and for its generalization with marked points $\mathcal{Y}_1(K)^n_m$.  The aim of this section is to prove the following.

\begin{theorem} \label{fondamentale} The stack $\mathcal{M}_{1,n}(BG, (h_1), \ldots, (h_n))$ admits a natural decomposition in components, not necessarily irreducible, that are parametrized by the subgroups $K$ of rank $0,1$ or $2$ of $E$. Let us fix a subset of the ramifications $\{h_1, \ldots, h_n\}$ that generates $R$ minimally; and call $k$ the cardinality of this subset. Given this choice, we construct an isomorphism from the component of the moduli stack  $\mathcal{M}_{1,n}(BG, (h_1), \ldots, (h_n))$ that corresponds to $K$, to an explicit complement of $(n-1) + \ldots+ (n-k)$ codimension-$1$ loci (not necessarily irreducible) inside an $E/K$-banded gerbe on $\mathcal{Y}_1(K)_{n-k}^k$.
\end{theorem}

The idea of the proof is to apply Proposition \ref{equivalentred2}, and then to ``trade'' each of the line bundles appearing in it with a point in $C$, by using the origin of the elliptic curve. So we obtain a moduli space of data $(C, x_1, \ldots, x_n, y_1', \ldots, y_{k}', y_1'' \ldots, y_{l}'')$; the compatibility conditions \eqref{equivalentred2} are translated into $k+l$ linear equivalences among the points (see \eqref{partesecondacondprim}). These linear conditions on the points $y_i''$ simply prescribe that they must be certain torsion points for the curve. The linear conditions for the points $y_i'$ (those obtained from the first line of \eqref{reducedsplit}), are more complicated. Up to possibly permuting the points $x_i$, and up to choosing a suitable basis of $R$, we show that the points $x_{n-k+1}, \ldots, x_{n}$ can be uniquely reconstructed from the rest of the data by means of these $k$ linear equivalences. Therefore, we can get rid of the points and of the linear conditions at the same time, to end up with an equivalent moduli space parametrizing $(C, x_1, \ldots, x_{n-k+1},y_1', \ldots, y_k', y_1'', \ldots, y_{l}'')$, where the $y_i''$ are torsion points of the elliptic curve $C$ generating a certain abelian group $K$. In this process the points $y_i$ need not be distinct nor different from the points $x_j$; on the other hand when reconstructing $x_{n-k+1}, \ldots, x_{n}$ up to isomorphism by using the linear equations, there is \emph{a priori} no guarantee that the points we reconstruct are \emph{pairwise distinct and different from} $x_1, \ldots, x_{n-k}$. The first observation says that the parameter space of $(C, x_1, \ldots, x_{n-k+1},y_1', \ldots, y_k', y_1'', \ldots, y_{l}'')$ is $\mathcal{Y}_1(K)_{n-k}^{k}$ (rather than $\mathcal{Y}_1(K)_{n}$). The second observation means that we have to exclude from this parameter space $(n-1) + \ldots+ (n-k)$ codimension-$1$ loci (not necessarily irreducible).

The following lemma will be crucial in the proof of Theorem \ref{fondamentale} and in the following sections, and it is one of the main reasons for working in the genus-$1$ case.
\begin{lemma} \label{closed} Let us fix an integer vector $\vec{a}=(a_1, \ldots, a_n)$ such that $\sum a_i=0$.
The condition
$ \sum_i a_i x_i  \equiv 0$
cuts out a codimension-$h$ closed substack inside $\mathcal{M}_{h,n}$.
\end{lemma}

\begin{proof} Consider any such family $C \to S$, and the line bundle on $C$ defined by \begin{displaymath} \mathcal{L}:= \mathcal{O}_C \left(\sum a_i x_i\right).
 \end{displaymath}
 The condition $\sum a_i x_i \equiv 0$ on a certain curve $C_s$ is equivalent to \begin{displaymath}h^0( C_s, \mathcal{L}_{C_s})\geq 1,\end{displaymath} and by semicontinuity this condition cuts out a closed locus on the base. The fact that the codimension is $h$ can be checked on geometric fibers.
\end{proof}

\begin{proof} (Theorem \ref{fondamentale}) We apply Proposition \ref{equivalentred}, with any fixed basis for $E^{\vee}= \langle e_1 \rangle \oplus \ldots  \oplus \langle e_{l} \rangle$. The elements in $G^{\vee}$ are chosen as follows. Up to permuting the elements $h_i$, we can assume that the subset of $k$ elements of $\{h_1, \ldots, h_n\}$ that generates $R$ consists of the last $k$ elements $\{h_{n-k+1}, \ldots, h_n\}$. We choose a basis $r_1, \ldots, r_k$ of $R$ that satisfies the conclusions of Lemma \ref{tecnico}. As elements of $G^{\vee}$, we choose any lift $g_1, \ldots g_k$ of $r_1^{\vee}, \ldots, r_k^{\vee}$ to $G^{\vee}$. So we obtain a description of the moduli stack \eqref{moduligenus1} as a moduli stack of objects $(C, x_j, L_i', L_i'', \phi_i', \phi_i'')$, constrained by the compatibility conditions \eqref{reducedsplit}.

We always have that $k \leq n$; if $k$ equals $n$, by Corollary \ref{tecnicopunto}, the set of solutions of the first set of compatibilities \eqref{reducedsplit} is empty, therefore the moduli stack is the empty set. Therefore, from now on we can assume that $k<n$, and fix $x_1$ as the origin of the elliptic curve $C$.

We begin by interpreting the line bundles $L_i'$ and $L_i''$ as points on the elliptic curve $(C,x_1)$. Each of the $L_i'$ has degree $\frac{\sum b_{i j}}{o(r_i)}$, and by twisting with the origin $x_1$
we have:
\begin{displaymath}
\exists! \ y_i' \in C, \quad L_i' \otimes \mathcal{O}_C\left(-\frac{\sum b_{i j}}{o(r_i)}x_1\right) \cong \mathcal{O}_C(y_i' -x_1),
\end{displaymath}
where the points $y_i'$ need not be pairwise disjoint and can coincide with the $x_j$. Similarly, each of the line bundles $L_i''$ corresponds to a point $y_i''$; the existence of the isomorphism $\phi_i''$ in the second half of \eqref{reducedsplit} imposes that \begin{displaymath}o(e_i) 	\cdot y_i''=0 \quad \textrm{in the elliptic curve } (C,x_1).\end{displaymath}
Associating each $L_i'$ with $y_i'$ and each $L_i''$ with $y_i''$ defines a functor from the groupoid of objects $(C, x_j, L_i', L_i'', \phi_i', \phi_i'')$ constrained by \eqref{reducedsplit}, to the groupoid whose objects are
\begin{equation} \label{parteseconda}
(C, x_1, \ldots, x_{n}, y_1', \ldots, y_k', y_1'', \ldots, y_l''),
\end{equation}
constrained by the linear equations
\begin{equation}\begin{cases} \label{partesecondacondprim}
o(r_i) \cdot y_i' \equiv \sum b_{i j} x_j + \sum c_{ij} y_j'' + (o(r_i)- \sum b_{i j}) x_1, & 1 \leq i \leq k; \\ o(e_i)\cdot y_i'' \equiv 0, & 1 \leq i \leq l.
\end{cases}\end{equation}
It is not difficult to see that this functor is fully faithful and essentially surjective, and hence an equivalence of categories.

We now work by induction on $k$, and begin by analyzing the case $k=0$.
Forgetting the points $x_2, \ldots, x_{n}$ gives a map to the stack of objects $((C, x_1), y_1'', \ldots, y_l'')$ with the conditions  $o(e_i) \cdot y_i''=0$. At most two among the points $y_i''$ can be independent in $\pic^0(C)$. So when we fix a subgroup $K$ of $E$, one component of this moduli stack parametrizes irreducible $K$-covers over the trivial $E/K$ cover of $C$. This moduli stack is an $E/K$-banded gerbe over $\mathcal{Y}_1(K)$.

Let us now consider the case $k>0$. We consider the forgetful map
\begin{equation}\label{variables}
f\colon (C, x_1, \ldots, x_n, \ldots, y_i',\ldots, y_j'', \ldots) \to (C,x_1, \ldots, x_{n-1}, \ldots, y_i', \ldots, y_j'', \ldots),
\end{equation}
which simply forgets the last marked point $x_n$. Since the data \eqref{parteseconda} satisfy \eqref{partesecondacondprim}, the isomorphism class of the point $x_n$ can be reconstructed from the rest of the data as a consequence of our initial choices of bases, indeed $b_{k n}$ equals $1$:
\begin{equation} \label{careful}
x_n \equiv o(r_k) \cdot y_k' - \sum_{j=1}^{n-1} b_{k j} x_j - \sum_{j=1}^l c_{k j} y_j'' - \left(o(r_k)- \sum_{j=1}^n b_{k j}\right) x_1.
\end{equation}
When we define the point $x_{n}$ through \eqref{careful}, we have to exclude the possibility that $x_{n}$ equals $x_j$ for each $j <n$. An object in the target of the functor $f$ \emph{does not} belong to the image of $f$ precisely when there exists a $1 \leq j < n$ such that
\begin{equation} \label{divisoris}
D_{j,n}: \quad o(r_k) \cdot y_k' - \sum_{s=1}^{n-1} b_{k s} x_s - \sum_{s=1}^l c_{k s} y_s'' - \left(o(r_k)- \sum_{s=1}^n b_{k s}\right) x_1 \equiv x_j.
\end{equation}
(We call this locus $D_{j,n}$, the fact that it is closed and of codimension $1$ follows from Lemma \ref{closed}).

The forgetful morphism $f$ identifies the stack of objects \eqref{parteseconda} with linear constraints \eqref{partesecondacondprim} with the stack of objects $(C, x_1, \ldots, x_{n-1}, y_1', \ldots, y_k', y_1'', \ldots, y_l'')$ with the \emph{open} condition given by the complement of the $D_{j,n}$, and the \emph{closed} conditions given by \begin{equation} \begin{cases} \label{partesecondacond1} o(r_i) \cdot y_i' \equiv b_{in} o (r_k) y_k' + \sum_{j=1}^{n-1} \tilde{b}_{i j} x_j  + \sum_1^n \tilde{c}_{ij} y_j'', &1 \leq i \leq k-1,\\ o(e_i)\cdot y_i'' \equiv 0, &  1 \leq i \leq l,
 \end{cases}
\end{equation}
where the coefficients are defined as \begin{displaymath} \begin{cases} \tilde{b}_{i1}:= b_{i 1}-b_{i n} b_{k 1}+ o(r_i)-b_{i n} o(r_k)- \sum_{s=1}^{n-1} \left( b_{i s}-b_{i n} b_{k s}\right),& j=1\\  \tilde{b}_{ij}:= b_{i j}-b_{i n} b_{k j}, & j> 1\\ \tilde{c}_{i j}:= c_{ij}-b_{in}c_{kj}. &\end{cases} \end{displaymath} This set of objects defines now a moduli stack of abelian $\tilde{G}:=G/\langle g_k \rangle$-covers of genus-$1$ curves. To see this, just apply Proposition \ref{equivalentred2} to \eqref{partesecondacond1} after having reduced the right hand side of the first set of equations modulo $o(r_i)$. This moduli stack of $\tilde{G}$-covers has ramifications $(\tilde{h}_1, \ldots, \tilde{h}_{n-1})$ defined by \begin{displaymath} \tilde{h}_j := h_j - b_{k j} \cdot h_n, \quad 1 \leq j \leq n-1, \end{displaymath} and the ramification subgroup of $\tilde{G}$ is $\tilde{R}:= R/\langle r_k \rangle$, whereas the \'etale part is left unchanged: $\tilde{E}=E$. We can now apply the induction hypothesis, since $\{\tilde{h}_{n-k+1}, \ldots \tilde{h}_{n-1}\}$ is a minimal generating set for $\tilde{R}$ of cardinality $k-1$.

The moduli stack defined in the paragraph above is thus the complement of the divisors $D_{j,n}$, inside the universal curve over the complement of $(n-2) + \ldots + (n-k)$ divisors in an $E/K$-banded gerbe over $\mathcal{Y}_1(K)_{n-k}^{k-1}$. The latter is a smooth moduli stack, as is \eqref{moduligenus1}, so the correspondence we have established is an isomorphism of stacks. This concludes the proof of the induction statement. \end{proof}

As a consequence of Theorem \ref{fondamentale}, we can enumerate the number of irreducible components of the loci of Lemma \ref{closed} when $h=1$, and hence all of the divisors $D_{j,n}$ that appear in \eqref{divisoris}.
\begin{corollary} When $h=1$, the codimension-$1$ locus defined in Lemma \ref{closed}, consists of $\phi(\gcd(a_1, \ldots, a_n))$ irreducible components, where $\phi$ is the Euler totient function. \label{componenti}
\end{corollary}
\begin{proof}
Fix any of the points $x_1, \ldots, x_n$, for example $x_n$, and define $\overline{a}_i$ as the reduction of $a_i$ modulo $a_n$ for all $1 \leq i <n$. Then any isomorphism $\mathcal{O}_C(-x_n)^{\otimes a_n} \cong \sum_{i<n} \overline{a}_i x_i $ defines a cyclic $\mathbb{Z}/a_n \mathbb{Z}$-cover over $C$ through Proposition \ref{equivalentred}. So the components of the locus defined by $\sum a_i x_i=0$ are in bijection with the components of the moduli space $\mathcal{M}_{1,n-1}(B \mathbb{Z}/a_n \mathbb{Z}, \overline{a}_1, \ldots, \overline{a}_{n-1})$, and we can apply Theorem \ref{fondamentale} to the latter. The \'etale part $E$ equals $\mathbb{Z}/d \mathbb{Z}$ where $d$ is $\gcd(a_1, \ldots, a_{n})$; we conclude that each irreducible component corresponds to a divisor of $d$.
\end{proof}
The last steps of the proof of Theorem \ref{fondamentale} can be made simpler (and induction can be avoided) in the following remarkable cases.

\begin{remark} \label{pgroup} The proof of Theorem \ref{fondamentale} becomes more transparent in the cases when $R$ is either a $p$-group, or a cyclic group. In these cases, we can apply Lemma \ref{tecnicop} (resp. Corollary \ref{tecnicocyc}) in place of Lemma \ref{tecnico} in the proof of Theorem \ref{fondamentale}. In this way, when we translate our moduli problem into the groupoid whose objects are \eqref{parteseconda}, constrained by the linear equations \eqref{partesecondacondprim}, the last $k$ columns of the matrix $b_{ij}$ form an upper triangular matrix, \emph{i.e.} with zeroes on the lower left triangular part, and ones on its diagonal. So, we can avoid the induction and consider the forgetful map $f^k$, which forgets $x_{n-k+1}, \ldots, x_n$ all in one go (cfr. \eqref{variables}):
\begin{equation}\label{variables2}
f^k\colon (C, x_1, \ldots, x_n, \ldots, y_i', \ldots, y_j'',\ldots) \to (C,x_1, \ldots, x_{n-k+1}, \ldots,  y_i', \ldots, y_j'', \ldots).
\end{equation}
We can then reconstruct $x_{i}$ for all $1 \leq i \leq k$, by means of the first line of \eqref{partesecondacondprim}
\begin{equation} \label{careful2}
x_{i} \equiv o(r_{i}) \cdot y_{i}' - \sum_{j=1}^{n-k+i-1} b_{k j} x_j - \sum_{j=1}^l c_{k j} y_j'' - \left(o(r_{i})- \sum_{j=1}^{n-k+i} b_{k j}\right) x_1.
\end{equation}
Again we must be careful that $x_i \neq x_j$ for $1 \leq j \leq n$ and $n-k < i \leq n$ (with $i \neq j$). This gives us $(n-1) + \ldots + (n-k)$ codimension-$1$ loci
\begin{equation} \label{divisoris2}
D_{j,i}: \quad o(r_i) \cdot y_i' - \sum_{s=1}^{n-k+i-1} b_{k s} x_s - \sum_{s=1}^l c_{k s} y_s'' - \left(o(r_i)- \sum_{s=1}^{n-k+i} b_{k s}\right) x_1 \equiv x_j
\end{equation}
that we have to exclude from the $E/K$-gerbe over $\mathcal{Y}_1(K)^{k}_{n-k}$ in order for it to be isomorphic with the component of $\mathcal{M}_{1,n}(BG; h_1, \ldots, h_n)$ that is singled out by the choice of $K$.
\end{remark}

The proof of Theorem \ref{fondamentale} strongly depends on Proposition \ref{equivalentred2}, and on the bases that we have fixed for $R^{\vee}$ and $E^{\vee}$. On the other hand, the statement of the theorem is formulated in a way that makes it independent of such choices. Our description is more explicit after a choice of the bases.

\begin{remark} Theorem \ref{fondamentale} exhibits each component of the moduli stack \eqref{moduligenus1} as an open substack of the fiber product
\begin{equation} \label{picture} \xymatrix{
 \mathcal{Y}_1(K)^k_{n-k} \ar[r] \ar[d]\ar@{}[rd]|{\square} & \mathcal{C}^k(\mathcal{M}_{1,n-k}) \ar[d]& \hspace{-0.7cm}\ni(C,x_1, \ldots, x_{n-k}, y_1, \ldots, y_k)\\
 \mathcal{Y}_1(K)_{n-k} \ar@{}[rd]|{\square}  \ar[r] \ar[d] & \mathcal{M}_{1,n-k} \ar[d]& \hspace{-2.6cm}\ni (C, x_1, \ldots, x_{n-k})\\
 \mathcal{Y}_1(K) \ar[r] & \mathcal{M}_{1,1}.&\\
}\end{equation}
In the \'etale case, when $h_1= \ldots=h_n=0$, $k$ equals $0$ and only the bottom square of \eqref{picture} is relevant. In the proof of Theorem \ref{fondamentale} one can stop at the $k=0$ step.
In the totally ramified case instead, only the right side of \eqref{picture} is relevant, as $K$ equals $0$. The proof of Theorem \ref{fondamentale} can be simplified a bit by eliminating the terms $L_i''$ and the corresponding $y_i''$.
\end{remark}

\begin{remark} When $h>1$, the problem of identifying the components of the moduli space \eqref{modulistack} is equivalent to the problem of identifying the components of $\mathcal{M}_{h,n}(B E; 0)$, where $E:= G/R$ as above. The latter problem is solved in the genus-$1$ case in the discussion following  \eqref{decomp}. In higher genus this problem is more complicated and is still open in general.

A remarkable set of cases when the latter problem is solved is when $E$ is a cyclic group and $\chara k =0$, see for example \cite[Theorem 1.2]{edmonds}. In our notation (where the total space of the covers needs not be connected), the irreducible components of $\mathcal{M}_{h,n}(B E; 0)$ are known to be in bijection with the subgroups of the cyclic group $E$.
\end{remark}

\subsection{Birational geometry} \label{birationalgeo}
From Theorem \ref{fondamentale}, we easily deduce a few birational-geometric properties of the coarse moduli space of
$\mathcal{M}_{1,n}(BG, h_1, \ldots, h_n)$ in the case of moduli of \emph{totally ramified} covers, see Definition \ref{totram}. We have already observed that, with this hypothesis, the moduli space is irreducible.

\begin{corollary} \label{birat1} The coarse space of \eqref{moduligenus1}, in the totally ramified case, is birational to $\mathcal{M}_{1,n}$.
\end{corollary}
In \cite{smyth1} and \cite{smyth2}, Smyth has studied modular compactifications of the moduli space $\mathcal{M}_{1,n}$. It would be interesting to establish a precise correspondence between the moduli spaces $\overline{\mathcal{M}}_{1,n}(BG, h_1, \ldots, h_n)$ and his birational models of $\overline{\mathcal{M}}_{1,n}$.

\begin{remark} Our birational equivalence is \emph{not} the forgetful morphism
\begin{equation}\label{nonbirat}
\mathcal{M}_{1,n}(BG, h_1, \ldots, h_n) \to \mathcal{M}_{1,n},
\end{equation}
instead, it is the restriction to a locally-closed subscheme of the right vertical map in \eqref{picture}. In the totally ramified case, we obtain a finite rational map from $\mathcal{M}_{1,n}$ to itself by composing the birational map of Corollary \ref{birat1} with \eqref{nonbirat}.

We now give another modular interpretation of such a finite rational map. For a fixed vector of integers $(a_1, \ldots, a_{n})$ whose sum is $1$, we can define a finite \emph{morphism} $\phi_{\vec{a}} \colon \mathcal{C}_{1,n-1} \to \mathcal{C}_{1,n-1}$ by the rule $(C, x_1, \ldots,  x_{n-1}, y) \mapsto (C, x_1, \ldots,  x_{n-1}, x)$; where $x$ is the unique point on $C$ such that $\mathcal{O}_C(x)$ is isomorphic to \begin{displaymath}\mathcal{O}_C\left( a_n y+\sum_{i<n}a_ix_i\right).\end{displaymath} (This is, in fact, an automorphism of $\mathcal{C}_{1,n-1}$ when $|a_n|= \pm 1$.) The morphism $\phi_{\vec{a}}$ induces a finite rational map from $\mathcal{M}_{1,n}$ to itself, where the latter is the open substack of $\mathcal{C}_{1,n-1}$ where $y \neq x_i$. The finite rational map mentioned at the beginning of this paragraph is obtained by subsequent composition of $k<n$ different rational maps $\phi_{\vec{a}}$.
\end{remark}

It is well-known that the moduli spaces $\overline{\mathcal{M}}_{1,n}$ are rational when $n \leq 10$, see \cite{belo}. The Kodaira dimension of the coarse moduli scheme of $\overline{\mathcal{M}}_{1,n}$, $n\geq 11$ was computed by Bini-Fontanari \cite[Theorem 3]{bf1}.

\begin{corollary} \label{birat2} Any compactification of the coarse moduli scheme parametrizing \emph{totally ramified} abelian $G$-covers of elliptic curves is rational when the number $n$ of branch points is $\leq 10$. The Kodaira dimension of $\overline{\mathcal{M}}_{1,n}(BG, h_1, \ldots, h_n),$
in the totally ramified case, is $0$ if $n=11$ and it is $1$ if $n \geq 12$.
\end{corollary}
For example, the compactification of the moduli space of genus-$g$ bielliptic curves with ordered ramification points is rational when $g \leq 6$, and it has Kodaira dimension $1$ when $g > 6$.
The rationality properties of the moduli subscheme of $\mathcal{M}_g$ of bielliptic curves (without a choice of the involution and without an ordering of the points) has been the object of recent study through more standard methods. In \cite{bace1}, \cite{cace1} and \cite{bace2} the authors prove the rationality of the latter moduli space in genus $g \leq 5$ (see also \cite{cace}); in \cite{cace2} the same moduli space is proven to be unirational also when $g\geq 6$.

\begin{remark} \label{drop} The hypothesis of total ramification cannot be removed from Corollary \ref{birat2}. Indeed, consider the case of \'etale $G$-covers. From Theorem \ref{fondamentale}, a component of the moduli space of \eqref{moduligenus1} is birational to the $(n-1)$-fold product of a certain modular curve $Y_1(K)$ with itself. The modular curve has genus $>1$ for a suitable choice of $K$, so the components corresponding to this $K$ in the moduli space of \eqref{moduligenus1} are of general type.
\end{remark}

\subsection{The Picard group of moduli of abelian covers of elliptic curves}

Our starting point in this section is the following result.

\begin{theorem} \label{vanishing} The rational Picard group of $\mathcal{M}_{1,n}$ vanishes. The rational Picard group of $\overline{\mathcal{M}}_{1,n}$ is \emph{freely} generated by the boundary divisors.
\end{theorem}
These are well-known facts. For the first part of the statement, in $\chara k=0$, this is a consequence of the vanishing of the rational cohomology of $\mathcal{M}_{1,n}$ in degrees $1$ and $2$, which is due to Harer \cite{harer1}. In the general case, the theorem follows from an elegant argument by Arbarello-Cornalba (see \cite{arbarellocornalba}, in particular Theorem 10). In this argument, one has to substitute (rational) singular cohomology with \'etale cohomology, and observe that \cite[Theorem 2]{arbarellocornalba} can be proven algebro-geometrically in $\chara k>0$ by using that the coarse moduli space of $\mathcal{M}_{1,n}$ is affine. The second part of the statement is proven in \cite[Proposition 1.9]{arbarellocornalba2}.

Here is our main result of this section.

\begin{theorem} \label{vanishes} The rational Picard group of the moduli stack \eqref{moduligenus1} vanishes in the totally ramified case. The boundary divisors in $\overline{\mathcal{M}}_{1,n}(BG, h_1, \ldots, h_n)$ have $r$ independent relations, where $r$ is the difference between the number of irreducible components of the codimension-$1$ loci mentioned in Theorem \ref{fondamentale}, and $(n-1)+ \ldots + (n-k)$.
\end{theorem}
\noindent Actually, the statement and the proof of the above result are the same if one takes $\overline{\mathcal{M}}_{1,n}(BG, h_1, \ldots, h_n)$ to be any smooth compactification of ${\mathcal{M}}_{1,n}(BG, h_1, \ldots, h_n)$.

 A consequence of Theorem \ref{vanishes} is that the Picard group of the moduli stack of totally ramified abelian covers of elliptic curves is a torsion group (with or without ordering the branch points).  The hypothesis of total ramification cannot be dropped: see Remark \ref{drop}. The number $r$ of relations can be computed inductively by using Corollary \ref{componenti} to compute the number of components of the loci \eqref{divisoris}.

To prove the theorem, we need some easy consequence of Theorem \ref{vanishing}. Let us introduce some notation. If $\mathcal{C}^k_{1,n-k}$ is the iterated $k$-th universal curve over $\mathcal{M}_{1,n-k}$, our convention is to call $x_1, \ldots, x_{n-k}$ the marked points and $y_i'$ the section corresponding to the $i$-th universal curve.  For $1 \leq i \leq k$, we define the ``boundary divisors of $\mathcal{M}_{1,n}$ in $\mathcal{C}^k_{1,n-k}$'':
\begin{equation} \label{bij}
B_{j,i}: \begin{cases} \left(  x_j \equiv y_i'\right) & \textrm{when } 1 \leq j \leq n-k, \\
\left( y_{k-n+j}' \equiv y_i' \right) & \textrm{when } n-k+1 \leq j \leq n, k-n+j \neq i.\end{cases}
\end{equation}
Note that $B_{n-k+j,i}$ equals $B_{n-k+i, j}$ for $1 \leq i \neq j \leq k$. So we have in total $(n-1) + \ldots + (n-k)$ distinct boundary divisors. We can identify the complement of these boundary divisors in $\mathcal{C}_{1,n-k}^k$ with $\mathcal{M}_{1,n}$, where the marked points of the latter are called $x_1, \ldots, x_n, y_1', \ldots, y_k'$. A corollary of Theorem \ref{vanishing} is:
\begin{corollary} \label{generated} The rational Picard group of $\mathcal{C}^k_{1,n-k}$ is freely generated by the classes $[B_{j,i}]$.
\end{corollary}

\begin{proof} (of Theorem \ref{vanishes}) Let us begin with the vanishing statement. We fix the elements of $G^{\vee}$ as in the first paragraph of Theorem \ref{fondamentale}: any basis $r_1, \ldots, r_k$ of $R^{\vee}=G^{\vee}$ constructed through Lemma \ref{tecnico}. Our proof is by induction on $k$, like the proof of Theorem \ref{fondamentale}, with the simplification that in this case $E$ equals $0$. When $k=0$ the statement follows from Theorem \ref{vanishing}: indeed the moduli stack \eqref{moduligenus1} equals $\mathcal{M}_{1,n}$. We have a description of the (irreducible) moduli stack \eqref{moduligenus1} as the complement of the $n-1$ loci $D_{j,n}$ defined in \eqref{divisoris}, inside the moduli stack of objects $(C, x_1, \ldots, x_{n-1}, y_1', \ldots, y_k')$ with constraints given by the equations \eqref{partesecondacond1}. The latter moduli stack is the complement of the loci $D_{j,n}$ inside the universal genus-$1$ curve over the moduli stack of totally ramified covers, which we call\footnote{This is \emph{not} the universal curve over the moduli stack $\mathcal{M}_{1,n-1}(BG', \tilde{h}_1, \ldots, \tilde{h}_{n-1})$.}
\begin{equation} \label{ind} \mathcal{C}_{1,n-1}(BG', \tilde{h}_1, \ldots, \tilde{h}_{n-1}),\end{equation} where $\tilde{h}_i$ is defined in the proof of Theorem \ref{fondamentale}, and $G'$ is the quotient $G/ \langle r_k \rangle$. The moduli stack \eqref{ind} contains \begin{equation}\mathcal{M}_{1,n}(BG', \tilde{h}_1, \ldots, \tilde{h}_{n-1}, 0)\label{ind2} \end{equation}  as the complement of the pull-back from $\mathcal{C}_{1,n-1}$ of the divisors $B_{j,n}$.  By the induction hypothesis, \eqref{ind2} has trivial rational Picard group, since it is a moduli stack of totally ramified abelian $G'$-covers with a minimal generating set  $\{h_{n-k+1}, \ldots, h_{n-1}\}$ of length $k-1$. Therefore, the rational Picard group of \eqref{ind} is generated by the $B_{j,n}$. By Corollary \ref{lincomb} and Lemma \ref{invert}, we can express each $[B_{j,n}]$ as a linear combination of the classes $[D_{j,n}]$ in the rational Picard group of \eqref{ind}. This concludes the proof of the vanishing part. The statement on the number of independent relations follows from Lemma \ref{picardcompl}. \end{proof}
\begin{lemma} \label{picardcompl} Let $X$ be a smooth Deligne-Mumford stack, and $\overline{X}$ a smooth compactification with $\overline{X} \setminus X$ consisting of $n$ boundary divisors. Assume that $X \subset Y$ is an open substack, that the rational Picard group of $Y$ has finite rank $m$, and that $Y \setminus X$ consists of $m+r$ boundary divisors. Then the rational Picard group of $X$ is finitely generated, and it has rank $n-r$.
\end{lemma}

We now prove a general result to compute some divisor classes on universal curves in genus $1$. The rational Picard group of $\mathcal{C}_{1,n-1}$ is generated by the classes of the divisors $B_{j,n}$ as we have already observed. Let $\vec{v}=(b_1, \ldots, b_{n-1})$ be a vector of integers and define $a:= \sum b_j$.
\begin{lemma} In $\pic_{\mathbb{Q}}(\mathcal{C}_{1,n-1})$, the class of the locus
\begin{displaymath}
D_{\vec{b}, a}: \quad \left(a y \equiv \sum_{j<n} b_j x_j\right),
\end{displaymath}
(closed and of codimension $1$ by Lemma \ref{closed}) can be expressed as
\begin{displaymath}
[D_{\vec{b}, a}]= a \sum_{j < n} b_j [B_{j,n}].
\end{displaymath}\label{lemmoso}
\end{lemma}
\begin{proof}
Consider the diagonal inclusion inside the fiber product of universal curves
\begin{equation} \label{picture2} \xymatrix{ \mathcal{C}_{1,n-1} \ar[r]^{\Delta}
 &\mathcal{C}^2_{1,n-1} \ar^{{\pi}_2}[r] \ar_{\pi_1}[d]\ar@{}[rd]|{\square} & \mathcal{C}_{1,n-1} \ar^{\pi}[d]\\
 &\mathcal{C}_{1,n-1}  \ar[r]  & \mathcal{M}_{1,n-1},\\
}\end{equation}
and define the line bundle \begin{displaymath}\mathcal{G}:=\mathcal{O}_{\mathcal{C}}\left(\sum_{j<n} b_j x_j\right).\end{displaymath} There is a natural evaluation map $ev$ from the bundle of sections of $\mathcal{G}$ to the $(a-1)$-jet bundle of $\mathcal{G}$:
\begin{displaymath}
ev\colon \mathcal{E}:=\pi_{1 *} \pi_2^* (\mathcal{G}) \to \mathcal{F}:=\pi_{1 *} \left(\pi_2^*(\mathcal{G}) \otimes \frac{\mathcal{O}_{\mathcal{C}^{k+1}}}{\Delta_{\mathcal{C}^k}} \right).
\end{displaymath}
The divisor $D_{\vec{b}, a}$ is the degeneration locus of $ev$, and thus we can apply Porteous' formula, to obtain
\begin{equation}
[D_{\vec{b}, a}]= (c_1(\mathcal{F})- c_1(\mathcal{E})) \cap [\mathcal{C}_{1,n}].
\end{equation}
The bundle $\mathcal{E}$ is the pull-back of a bundle from $\mathcal{M}_{1,n-1}$, and therefore it has zero first Chern class by Theorem \ref{vanishing}. On the other hand, it is a general fact that the bundle $\mathcal{F}$ has the same Chern classes as
\begin{displaymath}
\mathcal{G} \otimes \left( \mathcal{O}_{\mathcal{C}} \oplus \Omega_{\pi} \oplus \ldots \oplus \Omega_{\pi}^{a-1} \right).
\end{displaymath}
This concludes the statement, after observing that $c_1(\mathcal{G}) \cap [\mathcal{C}_{1,n}]$ equals by definition $\sum_{j < n} b_j [B_{j,n}]$, and that $\Omega_{\pi}$ is trivial since $\pi$ is a family of genus-$1$ curves.
\end{proof}

 In \eqref{divisoris} we defined the loci $D_{j,n}$. From now on, as we are working in the totally ramified cases, we always have that the $D_{j,n}$ are codimension-$1$ substack of $\mathcal{C}_{1,n-1}$ (the case when $l=0$ in \eqref{divisoris}).
 We are now in a position to calculate the class of $[D_{j,n}]$ as a linear combination of the  classes $[B_{i,n}]$ in $\mathcal{C}_{1,n-1}$. This is a straightforward consequence of Lemma \ref{lemmoso}, where we take the point $y_k'$ as the parameter for the universal curve.

\begin{corollary} \label{lincomb} The following relations hold in the rational Picard group of $\mathcal{C}_{1,n-1}$
\begin{displaymath}
[D_{j,n}] = o(r_k) \left( [B_{j,n}]+  \sum_{i=1}^{n-1} b_{k i} [B_{i,n}]  + \left(o(r_k)- \sum_{i=1}^n b_{k i}\right) [B_{1,n}]\right).
\end{displaymath}
\end{corollary}
That $[D_{j,n}]$ can be expressed as a linear combination of the $[B_{i,n}]$ is a consequence of the fact that the latter generate the rational Picard group of $\mathcal{C}_{1,n-1}$. What we want here is to show that the explicit such expression can be inverted to express $[B_{j,n}]$ as a rational linear combination of the classes $[D_{i,n}]$.
\begin{lemma} \label{invert} The matrix of Corollary \ref{lincomb}, expressing the classes of the codimension-$1$ loci $[D_{j,n}]$ in terms of the $[B_{i,n}]$,  is invertible over $\mathbb{Q}$.
\end{lemma}

\begin{proof} As we work over the rationals, we can safely divide each line by $o(r_k)$. The resulting matrix is now obtained by perturbing, with $+1$ on the diagonal, the $(n-1) \times (n-1)$ matrix with equal rows, whose $i$-th element is
\begin{displaymath}
\begin{cases} o(r_k)- \sum_{s=1}^n b_{k s}+ b_{k 1} & i=1, \\ b_{k i} & i >1.
\end{cases}
\end{displaymath}
The sum of the elements in each row equals $o(r_k) \neq 0$. This completes the proof. \end{proof}

To conclude the section, let us see how the proof of Theorem \ref{vanishes} can be simplified, in the cases when $G=R$ is either cyclic or a $p$-group.

\begin{remark} \label{remarkp}(Cf. Remark \ref{pgroup}) When $G$ is either a cyclic group or a $p$-group, we have expressed the moduli stack $\mathcal{M}_{1,n}(h_1, \ldots, h_n)$ as the complement of the divisors $D_{j,i}$ inside $\mathcal{C}^k_{1,n-k}$, which in turn contains $\mathcal{M}_{1,n}$ as the complement of the divisors $B_{j,i}$ defined in \eqref{bij}. We are therefore adding $(n-1) + \ldots+ (n-k)$ divisors to $\mathcal{M}_{1,n}$, and then removing the same number of codimension-$1$ loci. In analogy with what is done in Corollary \ref{lincomb} and Lemma \ref{invert} we can express the classes of the $D_{j,i}$ in terms of the classes of the $B_{j,i}$, and then prove that the resulting matrix is invertible, thus concluding the proof of Theorem \ref{vanishes} more directly.
\end{remark}

\section{Moduli spaces of admissible bielliptic curves} \label{modulibiell}
In this section we summarize and recollect our results in the important example of moduli of bielliptic curves of genus $g>1$. By Riemann-Hurwitz, we have $2g-2>0$ ramification and branch points. The starting point of this section is this straightforward Corollary of Theorem \ref{vanishes}.
\begin{corollary} \label{linindip} The boundary strata classes in the Picard group of the moduli stack
\begin{equation} \label{modulistable}
\overline{\mathcal{M}}_{1,n}( B \mathbb{Z}/ 2\mathbb{Z}; \underbrace{1, \ldots, 1}_{(2g-2)}, \underbrace{0, \ldots, 0}_{(n-2g+2)})
\end{equation}
are linearly independent when $g \neq 2$ and they have one relation when $g=2$.
\end{corollary}
In \cite[(6.11)]{pagani2}, we have already observed the existence of this bielliptic relation in genus $2$, and we have written it down explicitly in terms of the boundary divisors.

Let us briefly review our description of moduli of bielliptic curves obtained in the previous section.
The open subscheme of \eqref{modulistable} that parametrizes smooth bielliptic curves, $\mathcal{M}_{1,n}( B \mathbb{Z}/ 2\mathbb{Z}; 1, \ldots, 1, 0, \ldots, 0)$,
  has been described by virtue of Theorem \ref{fondamentale} as the complement of the $n-1$ divisors
\begin{displaymath}
D_{j,n}: \quad 2 y - \sum_{i=1}^{2g-3} x_i + \left(2g- 4\right) x_1 - x_j \equiv 0
\end{displaymath}
in the universal curve $\mathcal{C}_{1,n-1}$ (parametrized by $y$) over $\mathcal{M}_{1,n-1}$.
By applying Corollary \ref{componenti}, we see that the loci $D_{j,n}$ are actually all irreducible, unless $g$ equals $2$, in which case there is only one locus $D_{12}$, which consists of two irreducible components; the statement of Corollary \ref{linindip} then follows from Theorem \ref{vanishes}.

Let us now fix $n=2g-2>0$. It is easy to see that there are $2^n-n$ boundary divisors in $\overline{\mathcal{M}}_{1,n}$, and they are independent by Theorem \ref{vanishing}. One divisor generically parametrizes  curves that are irreducible and of geometric genus $0$; we call this divisor $\Delta_0$. The other divisors are indexed by subsets $I$ of  $[n]$ with $|I|\geq2$. Each such divisor $\Delta_I$ generically parametrizes curves with one node that separates a smooth genus $0$ component with marked points in the set $I$ from a smooth genus $1$ component with the remaining marked points $[n] \setminus I$. Let us analyze the inverses of these divisors under the forgetful map
\begin{displaymath}
\overline{\mathcal{M}}_{1,n}( B \mathbb{Z}/ 2\mathbb{Z}; 1, \ldots, 1) \to \overline{\mathcal{M}}_{1,n}.
\end{displaymath}
The fiber over $\Delta_{[n]}$ consists of two irreducible components, we call $\Theta_{[n]}$ the component where the double cover of the elliptic curve is non-trivial, and $\Theta_*$ the component where the double cover is trivial. The fiber over $\Delta_0$ also splits in two components: one where the node is generically a branch point of the cover, which we call $\Delta_0^{br}$, and  one where the node is a smooth point, which we call $\Delta_0^{et}$. The fiber over $\Delta_I$ consists of one component when $I \neq [n]$, we call this component $\Theta_I$ when $|I|$ is even, and $\Xi_I$ when $|I|$ is odd.  When $|I|$ is even the generic fiber is obtained by gluing a hyperelliptic curve of genus $g_1$ with a bielliptic curve of genus $g_2$ at two orbits on each curve under the distinguished involution. In this case, we have that $g_1+g_2+1=g$ and $|I|=2g_1+2$. On the other hand, when $|I|$ is odd, the generic fiber is obtained by gluing a hyperelliptic curve of genus $g_1$ with a bielliptic curve of genus $g_2$ at two ramification points. Here we have $g_1+g_2=g$ and $|I|+1= 2g_1+2$. In Figure \ref{figura} we give a picture of the general element of the boundary divisors.

\begin{figure}[ht]
\centering
\tiny
\psfrag{2}{$\mathcal{H}_{g-2}$}
\psfrag{1}{$\mathcal{H}_{g-1}$}
\psfrag{7}{$\mathcal{B}_{g-2}$}
\psfrag{6}{$\mathcal{M}_{1,1}$}
\psfrag{5}{$$}
\psfrag{3}{\hspace{-0.2cm}$\mathcal{H}_{g_1}$}
\psfrag{4}{$\mathcal{B}_{g_2}$}
\psfrag{8}{$\hspace{-0.1cm}\mathcal{H}_{g_1}$}
\psfrag{9}{$\mathcal{B}_{g_2}$}

\begin{tabular}{ccccccccc}
\includegraphics[scale=0.2]{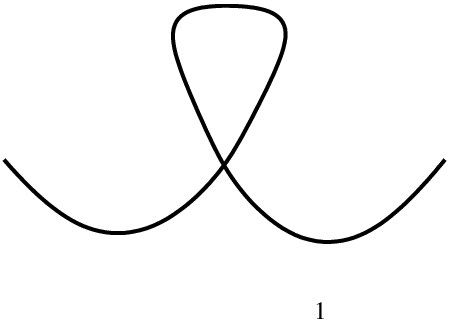}&&
\includegraphics[scale=0.2]{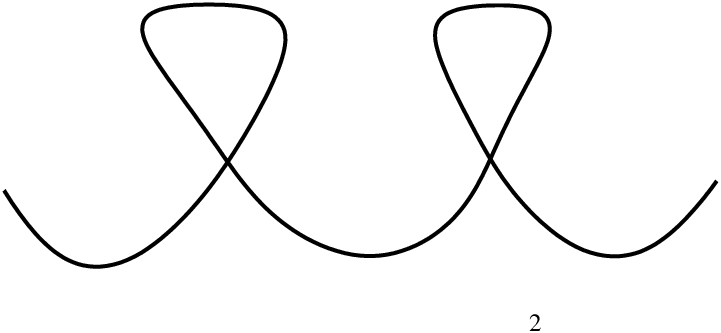}&&
\includegraphics[scale=0.2]{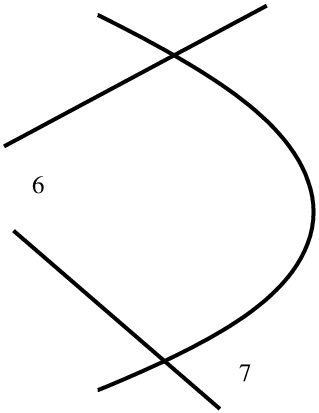}&&
\includegraphics[scale=0.2]{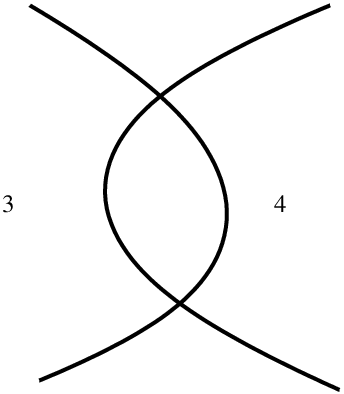} &&
\includegraphics[scale=0.2]{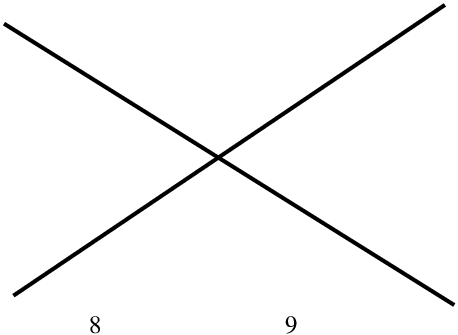}
\\
\end{tabular}
\caption{\label{figura} The general element of $\Delta_0^{br}, \Delta_0^{et}, \Theta_*, \Theta_{[I]}, \Xi_{[I]}$ respectively.}
\end{figure}

Let us denote with $\overline{\mathcal{B}}_g$ the stack quotient of \eqref{modulistable} via the natural action of the symmetric group $\mathfrak{S}_n$; the invariant boundary divisors are $n$ in the case of $\overline{\mathcal{M}}_{1,n}$, and $2g$ in the case of $\overline{\mathcal{B}}_g$. We have the following explicit numerical results on the rational Picard groups, where $\delta_{2,g}$ is the Kronecker delta.

\begin{corollary} \label{numerico} The Picard number of $\overline{\mathcal{M}}_{1,n}(B \mathbb{Z}/ 2\mathbb{Z}; 1, \ldots, 1)$ is $2^n+2-n-\delta_{2,g}$, the Picard number of $\overline{\mathcal{B}}_g$ equals $2 g- \delta_{2,g}$.  \end{corollary}

One can perform a similar study of the boundary components of the moduli stack  of admissible double covers of curves of genus $h >1$. As one can easily see, there are $g(h+1) -2h (h+1) + \lfloor \frac{ 5h +6}{2} \rfloor$ irreducible boundary divisors, when the branch points are unordered. Of course in higher genus the rational Picard group of the open moduli spaces will no longer be trivial. Still, we believe it is reasonable to expect that the open part should be generated by some decoration of the $\lambda$-classes, which should only depend on $h$ and not on $g$.

\begin{conjecture} The Picard number of moduli of admissible double covers of curves of genus $h>1$ is $(h+1)g + g \cdot o(1)$ for $g \to \infty$.
\end{conjecture}

In conclusion, let us remark that it is not simple to write down a closed formula for the number of relations among the boundary divisors of the totally ramified components of $\overline{\mathcal{M}}_{1,n}(BG)(h_1, \ldots, h_n)$, even when $G$ is cyclic. One treatable case is when $G$ is a $p$-group. For example, when $G=\mathbb{Z}/p^{\alpha}\mathbb{Z}$, we can choose $\langle g \rangle = G$ and write $h_i= \sum a_i g_i$, assuming wlog that $a_n=1$. Then there are no relations when $\gcd(a_1, \ldots, a_{n-1}, p^{\alpha})$ equals $1$; if $\gcd(a_1, \ldots, a_{n-1}, p^{\alpha})$ equals $p^{\beta}$, the space of relations has dimension $\phi(p^{\beta})=p^{\beta}(1- 1/p)$.

\section{Appendix - Some results on finite abelian groups}

We now prove a series of lemmas in the theory of finite abelian groups, which we need in the main body of the paper, but for which we could not find any reference. We use the same notation as in the main body, for example we denote with $R$ the group.

 When $R$ is a finite abelian $p$-group, with a given set of minimal generators, we can find a basis of $R$ such that the matrix that represents the generators in terms of the basis is triangular.

\begin{lemma} \label{tecnicop} Let $R$ be a finite abelian $p$-group, and $\{h_1, \ldots, h_k\}$ a minimal generating subset. Then up to reordering the elements $h_j$, we can find a basis $R=\langle r_1\rangle \oplus \ldots \oplus \langle r_{k}\rangle$, such that, for all $0 \leq i < k$, the projections of $h_{k-i}$ and of $r_{k-i}$ in the quotient $R/ \langle r_k, \ldots, r_{k-i+1}\rangle$ are equal.
\end{lemma}
\begin{proof} To construct the basis $r_1, \ldots, r_k$, we use the following fact: with the hypothesis of the lemma, an element of maximal order generates a direct summand of the group. There must be an element of maximal order among the elements $h_i$ (because they are generating) and this element must be unique (because they are minimal). We assume that this element is $h_n$ and define $r_k:=h_n$. The element $r_k$ generates a direct summand in $R$, so we consider $\{h_1, \ldots, h_{k-1}\}$: it is a minimal generating set for the quotient $R/\langle r_k \rangle$. Again we must have an element, assume it is $h_{k-1}$, whose image in the quotient has maximal order. So we define $r_{k-1}:= h_{k-1}$, and so on. The process is concluded in precisely $k$ steps, otherwise the starting set $\{h_1, \ldots, h_k\}$ would either be not generating or not minimal. (This proves, in particular, that for a $p$-group $R$ the cardinality of a minimal generating subset is a well-defined integer $k$).
\end{proof}

In the case of a general finite abelian group $R$ and a minimal set of generators, it is not always possible to find a basis such that the matrix representing the generators in terms of the basis is triangular. Anyway, following the idea of the proof of Lemma \ref{tecnicop}, we can prove the following, which is used in Theorem \ref{fondamentale}.
\begin{lemma} \label{tecnico} Let $R$ be a finite abelian group, and $\{h_1, \ldots, h_k\}$ a minimal generating subset. Up to reordering the elements $h_j$, there exists a basis with $k$ elements of $R=\langle r_1\rangle \oplus \ldots \oplus \langle r_{k}\rangle$ such that for all $0 \leq i <k$,
\begin{enumerate}
\item The image $\bar{h}_{k-i}$ of $h_{k-i}$ in $R/\langle r_k,\ldots, r_{k-i+1}\rangle$ equals $1 \cdot r_{k-i}+ \ldots$.
\item The morphism induced on the quotient by the componentwise inclusion
\begin{displaymath}
\langle h_1\rangle \oplus \ldots \oplus \langle h_{k-i-1} \rangle \to R/\langle r_k, \ldots, r_{k-i} \rangle
\end{displaymath}
is surjective.
\end{enumerate}
\end{lemma}

\begin{proof}
 Let us decompose $R$ in $p$-groups:
\begin{displaymath}
R^{(1)}:=R= = R_{p_1}^{(1)} \oplus \ldots \oplus R_{p_{s_1}}^{(1)},
\end{displaymath}
and call $\pi_i^1$ the projection onto the $i$-th factor. There exists $\sigma_1 \colon [s_1] \to [k]$, such that for all $i \in [s_1]$, when $\sigma_1(i)=j$, $\pi_i^1(h_j)$ has maximal order in $G_{p_i}$. For all $j \in \sigma_1([s_1])$, we define \begin{displaymath} r_j:= \sum_{ \sigma_1(i)=j} \pi_i^1(h_{j}).\end{displaymath}  Being an element of maximal order, each $r_j$ generates a direct summand of $R$.

Now we call $R^{(2)}$ the quotient of $R=R^{(1)}$ by the group generated by the $r_j$ for all $j \in \sigma_1([s_1])$, and consider the minimal set of generators of $R^{(2)}$ chosen as a subset of $\{h_j\}$ with $j \in [k] \setminus \sigma_1([s_1])$. We can apply the same procedure, and decompose $R^{(2)}$  in $p$-groups
\begin{displaymath}
R^{(2)} = R_{p_1}^{(2)} \oplus \ldots \oplus R_{p_{s_2}}^{(2)}.
\end{displaymath}
Let $\pi_i^2$ be the projection onto the $i$-th factor: there exists $\sigma_2 \colon [s_2] \to [k]$, with $\sigma_2([s_2])$ disjoint with $\sigma_1([s_1])$ and such that, for all $i \in [s_2]$, $\pi_i^{(2)}(h_j)$ has maximal order in $R^{(2)}$ when $\sigma_2(i)=j$. Again we define, for all $j \in \sigma_2([s_2])$, \begin{displaymath} r_j:= \sum_{ \sigma_1(i)=j} \pi_i^2(h_{j}).\end{displaymath}

We continue this procedure, until it produces the requested basis $r_1, \ldots, r_k$ in a finite number $t$ of steps, such that $k_1 + \ldots + k_t=k$. (This also proves in particular that $k$ must be greater than or equal to the maximum of the ranks of the groups $R_{p_i}$, and smaller than or equal to the sum of the ranks of the groups $R_{p_i}$.)
\end{proof}

From the proof of Lemma \ref{tecnico}, it is evident that the same triangularity result obtained in Lemma \ref{tecnicop} holds when $R$ is a cyclic group.
\begin{corollary} \label{tecnicocyc} Let $R$ be a finite cyclic group, and $\{h_1, \ldots, h_k\}$ a minimal generating subset. Then we can find a basis $R=\langle r_1\rangle \oplus \ldots \oplus \langle r_{k}\rangle$, such that, for all $0 \leq i < k$, the projections of $h_{k-i}$ and of $r_{k-i}$ in the quotient $R/ \langle r_k, \ldots, r_{k-i+1}\rangle$ coincide.
\end{corollary}
From the construction of Lemma \ref{tecnico}, we also deduce the following corollary, which is crucial in the proof of Theorem \ref{fondamentale} to show that we must have an extra marked point, which we can then fix as the origin of the elliptic curve $C$.
\begin{corollary} \label{tecnicopunto} Let $R$ be a finite abelian group, and $\{ h_1, \ldots, h_k \}$ a minimal set of generators  and $R=\langle r_1\rangle \oplus \ldots \oplus \langle r_{k}\rangle$ be the basis constructed in Lemma \ref{tecnico}, and write $h_j= \sum_i b_{ij} r_i$, with $0 \leq b_{ij} < o(r_i)$. Then the equations \begin{displaymath} \sum_j b_{{i}j} \equiv 0 \mod (r_i)\end{displaymath} are not satisfied for all $1 \leq i \leq k$.
\end{corollary}
\begin{proof}
We prove that the equations are not satisfied for the last $k_t$ elements of the basis we have defined in the proof of Lemma \ref{tecnico}. Let us take one such $h$, and decompose $R$ as a direct sum of cyclic groups of prime power order, generated by certain elements. There is at least one of these elements, let us call it $s$, such that the coordinate of $h$ with respect to $s$ is coprime with $o(s)$, a power of a prime. By the construction of $h$,  the coefficients of all other elements $h_i$ with respect to $s$ are not coprime with $o(s)$: this concludes the proof.
\end{proof}

\begin{center}\textsc{Acknowledgments}
\end{center}
During the preparation of this manuscript, we have benefited from fruitful discussions with Eduardo Esteves, Carel Faber, Barbara Fantechi, Klaus Hulek, Michael L\"onne, Dan Petersen, Flavia Poma, Orsola Tommasi, and Fabio Tonini. The author was partly supported by the DFG project Hu 337/6-2.

\end{document}